\documentclass{article}
\usepackage{act-fouche}
\usetikzlibrary{positioning,decorations.markings,arrows.meta,calc,fit,quotes,cd}
\tikzset{
	oriented WD/.style={
			every to/.style={out=0,in=180,draw},
			label/.style={
					font=\everymath\expandafter{\the\everymath\scriptstyle},
					inner sep=0pt,
					node distance=2pt and -2pt},
			semithick,
			node distance=1 and 1,
			decoration={markings, mark=at position .5 with {\arrow{stealth};}},
			ar/.style={postaction={decorate}},
			execute at begin picture={\tikzset{
						x=\bbx, y=\bby,
						every fit/.style={inner xsep=\bbx, inner ysep=\bby}}}
		},
	bbx/.store in=\bbx,
	bbx = 1.5cm,
	bby/.store in=\bby,
	bby = 1.75ex,
	bb port sep/.store in=\bbportsep,
	bb port sep=2,
	bb port length/.store in=\bbportlen,
	bb port length=4pt,
	bb min width/.store in=\bbminwidth,
	bb min width=1cm,
	bb rounded corners/.store in=\bbcorners,
	bb rounded corners=2pt,
	bb small/.style={bb port sep=1, bb port length=2.5pt, bbx=.4cm, bb min width=.4cm, bby=.7ex},
	bb/.code 2 args={
			\pgfmathsetlengthmacro{\bbheight}{\bbportsep * (max(#1,#2)+1) * \bby}
			\pgfkeysalso{draw,minimum height=\bbheight,minimum width=\bbminwidth,outer sep=0pt,
				rounded corners=\bbcorners,thick,
				prefix after command={\pgfextra{\let\fixname\tikzlastnode}},
				append after command={\pgfextra{\draw
							\ifnum #1=0{} \else foreach \i in {1,...,#1} {
								($(\fixname.north west)!{\i/(#1+1)}!(\fixname.south west)$) +(-\bbportlen,0) coordinate (\fixname_in\i) -- +(\bbportlen,0) coordinate (\fixname_in\i')}\fi
							\ifnum #2=0{} \else foreach \i in {1,...,#2} {
								($(\fixname.north east)!{\i/(#2+1)}!(\fixname.south east)$) +(-\bbportlen,0) coordinate (\fixname_out\i') -- +(\bbportlen,0) coordinate (\fixname_out\i)}\fi;
						}}}
		},
	bb name/.style={append after command={\pgfextra{\node[anchor=north] at (\fixname.north) {#1};}}}
}

\usepackage{authblk}

\providecommand{\event}{ACT2023}
\usepackage{hyperref}
\hypersetup{ colorlinks=true
	, linkcolor=black
	, urlcolor=black
	, citecolor=black
}
\usepackage{moirai}

\NewDocumentCommand{\vxy}{o m}{
	\IfNoValueTF{#1}
	{\vcenter{\xymatrix{#2}}}
	{\vcenter{\xymatrix#1{#2}}}
}
\NewDocumentCommand{\aMre}{o}{
	\IfNoValueTF{#1}
	{\mathsf{s}\Mre}
	{\mathsf{s}^{#1}\Mre}
}
\def\id{\text{id}}

\def\un{\sfu}
\DeclareMathOperator{\forg}{U}
\DeclareMathOperator{\free}{F}

\definecolor{refkey}{HTML}{ADD8E6}
\definecolor{labelkey}{HTML}{ADD8E6}

\def\h{\textsf{head}}
\def\kons{\mathbin{::}}
\def\concat{\mathbin{\scalebox{.8}{$+\kern-.2em+$}}}

\def\mto{\rightarrowtriangle}

\newcommand{\Pageref}[1]{p.~\pageref{#1}}
\newcommand{\deferredRef}[1]{The proof is deferred to the appendix, \Pageref{proof_of_#1}}

\NewDocumentCommand{\inMe}{m m m}{\mathop{\dsdtstile{#2,#3}{#1}}}
\NewDocumentCommand{\inMo}{m m m}{\mathop{\dsststile{#2,#3}{#1}}}

\NewDocumentCommand{\LinMe}{O{A} O{B} m m m}{\bsmat #1\esmat\dsdtstile{#4,#5}{#3}\bsmat #2\esmat}
\NewDocumentCommand{\LinMo}{O{A} O{B} m m m}{\bsmat #1\esmat\dsststile{#4,#5}{#3}\bsmat #2\esmat}

\newcommand{\xyadj}[4]{\xymatrix{
#1 \ar@<.5em>[r]^-{#3}\ar@{}[r]|-\perp&\ar@<.5em>[l]^-{#4} #2
}}

\newcommand{\inlinexyadj}[4]{\xymatrix{
		#3 : #1 \ar@<.33em>[r] \ar@{}[r]|-\perp&\ar@<.33em>[l] #2 : #4
	}}

\def\mlyExt{{\text{e}+}}
\def\mreExt{{\text{e}*}}

\usepackage[
	backend=bibtex,
	style=alphabetic,
]{biblatex}

\def\baseRepo{https://github.com/iwilare/categorical-automata}
\def\blobRepo{\baseRepo/blob/main/}
\newcommand{\Href}[2]{\href{#1}{\textsf{#2}}}
\newcommand{\agda}[1]{%
	\hspace{0.07em}%
	\Href{\blobRepo #1}{(\raisebox{-0.2em}{%
			\includegraphics[height=1em]{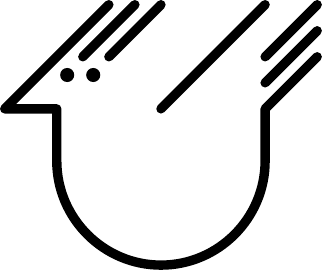}%
		})%
		\hspace{0.20em}}%
}
\newcommand{\agdasymbol}[1]{%
	\hspace{0.07em}%
	\Href{\baseRepo}{\raisebox{-0.2em}{%
			\includegraphics[height=1em]{agda}%
		}%
		\hspace{-0.40em}%
  }
}

\usepackage[left=3cm,right=3cm]{geometry}
\newcommand{\email}[1]{\href{mailto:#1}{\tt #1}}
\begin{document}
\title{The semibicategory of Moore automata}
\author[1]{Guido \textsc{Boccali}}
\author[2]{Bojana \textsc{Femi\'c}}
\author[3]{Andrea \textsc{Laretto}}
\author[3]{Fosco \textsc{Loregian}}
\author[4]{Stefano \textsc{Luneia}}
\affil[1]{Università di Torino, Torino, Italy\newline \email{guidoboccali@gmail.com}}
\affil[2]{Serbian Academy of Sciences and Arts, Belgrade, Serbia\newline \email{femicenelsur@gmail.com}}
\affil[3]{Tallinn University of Technology, Tallinn, Estonia\newline \email{anlare@ttu.ee},\email{folore@ttu.ee}}
\affil[4]{Università di Bologna, Bologna, Italy\newline \email{stefano.luneia@gmail.com}}
\maketitle

\begin{abstract}
	We study the semibicategory $\Mre$ of \emph{Moore automata}: an arrangement of objects, 1- and 2-cells which is inherently and irredeemably nonunital in dimension one.

 Between the semibicategory of Moore automata and the better behaved bicategory $\Mly$ of \emph{Mealy automata} a plethora of adjunctions insist: the well-known essential equivalence between the two kinds of state machines that model the definitions of $\Mre$ and $\Mly$ is appreciated at the categorical level, as the equivalence induced between the fixpoints of an adjunction, in fact exhibiting $\Mre(A,B)$ as a coreflective subcategory of $\Mly(A,B)$; the comodality induced by this adjunction is but the $0$th step of a `level-like' filtration of the bicategory $\Mre$ in a countable family of essential bi-localizations $\aMre[n]\subseteq\Mre$. We outline a way to generate intrinsically meaningful adjunctions of this form. We mechanize some of our main results using the proof assistant \texttt{Agda}.
\end{abstract}

\section{Introduction}\label{sec_intro}
Any arrangement of sets and functions $E \xot d A\times E, E\xto s B$
can be interpreted as follows:
\begin{itemize}
	\item $E$ as a set of `states', and $A,B$ are sets of `inputs' and `outputs';
	\item the function $d : E\times A \to E$ is a \emph{transition function} (specifying a \textbf{d}ynamics), and the function $s : E\to B$ provides an \emph{output} given an initial state (specifying a final \textbf{s}tate).
\end{itemize}
Altogether, $(d,s)$ form a \emph{Moore automaton} or Moore `machine'.

This definition extends to the case where $E,A,B$ are objects of the same monoidal category $(\clK,\otimes)$ \cite{Ehrig}; the situation is of particular interest when $\clK$ is Cartesian or semiCartesian: for example $\clK=(\TTop,\times)$ (and then $d,s$ are continuous, and in particular we can think of $f$ as an action over $E$ of the free topological monoid $A^*$ on $A$, cf. \cite{reutenauer1979topologie}),\footnote{It is worth to remark that here `topological' can --and should-- be intended broadly as follows: let $\clE$ be a Grothendieck topos; we can consider Moore automata in $\clE$, and then $E$ is an object of $\clE^M$, the category of $\clE$-objects with an action of $M=A^*$, the free monoid on $A$. Cf. also \cite{rogers2021toposes} for a thorough study of `toposes of monoid actions' when $\clE=\Set$.} but also $\clK=(\Cat,\times)$ (and then $f,g$ are functors, cf. \cite{noi:bicategories,guitart1974remarques,guitart1978bimodules}), $\clK=\Par$ (the category of sets and partial functions; in this case, $f,g$ are `partially defined'  functions that can fail on certain inputs), or a category of smooth manifolds \cite{spivak2016steady}.

For fixed input and output $A,B$ the structure of the class $\Mre(A,B)$ of Moore automata can be described as the strict pullback
\[\label{as_pb}\vxy{
	\Mre(A,B) \ar[r]^-{V'}\ar[d]_{U'}\xpb & \clK/B\ar[d]^-{V}\\
	\Alg(A\otimes \firstblank) \ar[r]_-{U} & \clK
	}\]
where $U : \Alg(A\otimes \firstblank) \to \clK$ is the forgetful functor from the category of endofunctor algebras for $A\otimes \firstblank : \clK\to \clK$, or --which is equivalent if $\clK$ has countable colimits-- actions of $A^*$, and $V : \clK/B\to \clK$ is the usual `domain' fibration \cite[p. 30]{CLTT}, \cite{loregian2019categorical} from the slice over $B$. Universal constructions are induced from the underlying category $\clK$ (in particular, all colimits and connected limits are created by $V'$, a result we survey in \cite{noi:completeness}).

As clean as it may seem, this characterisation leaves something out of the story; thinking of objects of $\Mre(A,B)$ as \emph{processes} from $A$ to $B$, it is natural to compose them \emph{sequentially}, as follows:
\[
	\begin{tikzpicture}[oriented WD, bbx=1em, bby=1ex, baseline=(current bounding box.center), scale=0.9]
		\node[bb={1}{1}, font=\tiny] (X1) {$(E,f,g)$};
		\node[bb={1}{1}, font=\tiny, right =2 of X1] (X2) {$(E',f',g')$};
		\node[bb={1}{1}, fit={($(X1.north west)+(-1,3)$) ($(X1.south)+(0,-3)$) ($(X2.east)+(1,0)$)}, bb name = ${\color{gray!70}(E\times E',f'\diamond f,g'\diamond g)}$, dotted, gray!70, font=\scriptsize] (Y) {};
		\draw[ar] (Y_in1') to (X1_in1);
		\draw[ar] (X1_out1) to (X2_in1);
		\draw[ar] (X2_out1) to (Y_out1');
		\draw[label]
		node at ($(Y_in1')!.5!(X1_in1)+(0,7pt)$)  {$A$}
		node at ($(X1_out1)!.5!(X2_in1)+(0,7pt)$)   {$B$}
		node at ($(X2_out1)!.5!(Y_out1')+(0,7pt)$)  {$C$};
	\end{tikzpicture}
\]
Given $(E,f,g)\in\Mre(A,B)$ and $(E',f',g')\in\Mre(B,C)$, we can consider the Moore automaton $(E\times E',f'\diamond f,g'\diamond g)\in\Mre(A,C)$, defined as
\[
	\begin{cases}
		f'\diamond f : E\times E' \times A \to E\times E' : (e,e',a)\mapsto (f(e,a),f'(e',ge)) \\
		g'\diamond g : E\times E'\to C : (e,e')\mapsto g'(e')
	\end{cases}
	\label{moore_compo}\]
This binary composition law is easily seen to be associative up to isomorphism. Now, do the various $\Mre(A,B)$ in \eqref{as_pb} arrange as the hom-categories of a bicategory $ \Mre $ of Moore automata? This would be the best situation to be in.

However, the composition $\firstblank\diamond\firstblank$ looks irredeemably nonunital, as there is no hope of finding a common section $A\mapsto \id_A$ for the domain and codomain endofunctors 
 sending $(E,f,g)\in\Mre(A,B)$ to $A$ and $B$ respectively. Indeed, to find such identity arrows, one should fix a triple $(U_A,x,y) : U_A\xot x A\times U_A, U_A \xto y A$ such that, at the very least, the compositions $(E\times U_B,\firstblank,\firstblank)$ and $(U_A\times E,\firstblank,\firstblank)$ are isomorphic to $(E,f,g)$ in $\Mre(A,B)$. So in particular, the carrier $U_A$ must be the terminal object / monoidal unit, which is easily seen to not work.

Confronted with this problem, two approaches are possible: one can forget about the composition law in \eqref{moore_compo}, and find one that  underlies a bicategory structure; or, one can forget about the absence of identity 1-cells, and study the structure exhibited by $\Mre$ `as it is'.

Leaning towards the first approach is not recommendable, as the composition law outlined above arises by restricting an existing bicategory structure. If we consider more general spans like
\[\vxy{E & A\times E \ar[r]^g\ar[l]_f& B,}\label{ex_mealy}\]
arranged in the class $\Mly(A,B)$, all the nice properties expounded above remain true, while at the same time addressing the problem: now that the $g$ is sensitive to the input, the identity 1-cell of an object $A$ is given by the span $1\xot ! A\times 1 \xto{\pi_A} A$,
given by the two projections. Moreover, this definition `makes sense', because when $\clK=\Set$ (or for that matter, any Cartesian category) the bicategory $\Mly$ so determined is just the category of \emph{Mealy automata}, which is known to be not just a bicategory, but \emph{the} bicategory \cite{Katis1997,ITA_2002__36_2_181_0,openTransitionSystems21} of pseudofunctors and lax natural transformations from $\bfN$ (the monoid of natural numbers) to $(\clK,\otimes)$: a bicategory of `lax dynamical systems', borrowing the analogy from \cite{Tierney1969}.
Our choice of composition of 1-cells is then constrained by the fact that $\firstblank\diamond\firstblank$ in $\Mly$ enjoys a certain universal property, and that in it composing 1-cells restricts to the subcategory of spans like \eqref{ex_mealy} with $g$ mute in the input variable.

We are thus led to follow the second path, and recognise that $\Mre$ is a \emph{semibicategory}: a structure that is precisely like a bicategory, except that it lacks identity 1-cells.

The purpose of this note is to initiate the study of $\Mre$ \emph{qua} semibicategory, embracing its (lack of) identity and trying to uncover its properties, both regarded as an object of the category of semibicategories, and related the `true' bicategory $\Mly$ where it embeds through a canonical functor $J$ (\autoref{tautj}).

In particular, assuming $\clK$ is Cartesian closed, we focus on the precise sense in which the semibicategory $\Mre$ and the bicategory $\Mly$, linked by $J$, fit into adjoint situations, and we provide a generic way of generating such adjunctions; leveraging on the pullback \eqref{as_pb}, and given a similar characterisation of $\Mly(A,B)$ as a strict pullback in $\Cat$, in \autoref{def_mlymre}.\ref{cf_1}, all adjunctions
\[\xyadj{\clK/B}{\Alg(A\times\firstblank)\cong\clK/B^A}{}{}\]
between slice categories pull back to $\Mre(A,B)$ and $\Mly(A,B)$. Now categorical algebra provides `standard' natural ways to obtain such adjunctions, when $\clK$ is rich enough (e.g., it is locally Cartesian closed): pulling back along a morphism or computing direct/inverse images \cite[A1.5.3]{Johnstone2002}, using the 2-functoriality of comma objects in $\Cat$ \cite{kelly_1989}, or the structure induced by the whole functor $\Mre(A,B)\to \clK/B\times\clK/B^A$ obtained from the pullback span \eqref{as_pb}. Each of these constructions yields a different adjunction, each of which is of some interest when translating 1-cells of $\Mre(A,B)$ and $\Mly(A,B)$ into `machines' that accept an input of type $A$ and produce an output of type $B$. In particular, we provide $J$ with a right adjoint, just to discover a moment later that $J=D_0$ is the $0$th term of a countable sequence of functors $D_n$, each of which is a fully faithful left adjoint $D_n\dashv K_n$ (cf. \eqref{di_enne}): this way, we obtain a rather intrinsic description for a classical equivalence theorem between Moore and Mealy machines \cite[3.1.4,3.1.5]{Shallit}, embedding $\Mre(A,B)$ in $\Mly(A,B)$ as a coreflective subcategory; in \autoref{the_univ_moo}, \autoref{moorif_right} we give a very explicit description of the coreflector as post-composition with a universal Moore cell $\fku X$. Adopting a similar point of view in \autoref{decap_right} we prove that a similar adjunction colocalises $\Mly(A,B)$ onto the subcategory of what we call \emph{soft} Moore machines (\autoref{def_acep}); again, the coreflector in \autoref{decap_right} is post-composition with a universal soft automaton $\fkp X$, that can be characterised (\autoref{constr_pinfty}, \autoref{P_is_acep}) as a certain pullback, or as the image of yet another functor (\autoref{some_adjoints_btwn_Mly}) obtained by a change of base/direct image adjoint situation \eqref{triple_for_precomp_funs}. This situation generalises: it is possible to define a whole hierarchy of subcategories $\aMre[n]$ for $n\ge 1$, each of which is the category of $D_nK_n$-coalgebras, thus arranged in a filtered diagram of coreflective inclusions
\[\vxy{
		\aMre\,\, \ar@{^{(}->}[r]& \,\,\aMre[2]\,\, \ar@{^{(}->}[r]& \,\,\aMre[3]\,\, \ar@{^{(}->}[r]& \,\,\dots\,\, \ar@{^{(}->}[r]& \,\,\Mre
	}\]
analogous in some regards to the `level' filtration of a topos \cite{kelly1989complete,menni2019level}, although the localisations in study here are not essential.

Unsurprisingly, a well-known and `solid' procedure like \cite[3.1.4,3.1.5]{Shallit} is just a universal construction in disguise. However, the interest in such a finding is also pedagogical: if hypothetical readers journeyed automata theory in any direction, armed only with bicategory theory, they would be able to rediscover natural constructions on state machines by just enacting what they know about more familiar universal constructions (composition of 1-cells, bicategorical limits \cite{kelly_1989,Street19,bird1989flexible}, especially of of Kan type,\footnote{Loosely speaking, a `bicategorical limit of a Kan type' is a left or right Kan extension, or a left/right Kan lifting, possibly pointwise or absolute.} \cite{street1981conspectus} the fundamental calculus of a Cartesian closed category, base change adjunctions\dots). The importance of such serendipity is hard to overestimate.

\section{The semibicategory of Moore machines}
Let's start by recording a precise definition of the categories of Moore and Mealy machines. Fix a monoidal category $(\clK,\otimes)$ in the background, which we will call the `base category'.
\begin{definition}[The hom-categories of $\Mly$ and $\Mre$] \agda{AsPullbacks.agda}\, \label{def_mlymre}
  Fix two objects $A,B\in\clK_0$; $A$ is to be thought of as the `input' object, and $B$ as the `output' object.

  Denote with $\Alg(A\otimes\firstblank)$ the category of endofunctor algebras for the functor $X\mapsto A\otimes X$ (or \emph{acts}, cf. \cite{kilp2000monoids}), and $(A\otimes\firstblank/B)$ the comma category of morphisms $A\otimes X\to B$. Then,
  \begin{enumtag}{cf}
    \item \label{cf_1} the category $\Mly(A,B)$ of Mealy machines with input $A$ and output $B$ is the strict pullback square
    \[\label{mly_pb}\vxy{
        \Mly(A,B) \ar[r]^{U'}\ar[d]_{V'} \xpb & (A\otimes\firstblank/B)\ar[d]^V \\
        \Alg(A\otimes\firstblank) \ar[r]_U & \clK
      }\]
    \item \label{cf_2} the category $\Mre(A,B)$ of  Moore machines with input $A$ and output $B$ is the strict pullback square
    \[\label{mre_pb}\vxy{
        \Mre(A,B) \ar[r]^{U'}\ar[d]_{V'} \xpb & \clK/B\ar[d]^V \\
        \Alg(A\otimes\firstblank) \ar[r]_U & \clK
      }\]
    if $\clK/B$ is the comma category of the identity functor $\id : \clK\to\clK$, and of the constant functor at $B\in\clK_0$.\footnote{It is worth noting three things: the definition of $\Mly(A,B)$ and $\Mre(A,B)$ depends on $\clK$; every lax monoidal functor $F:\clH\to\clK$ induces base changes pseudofunctors $F_* : \Mly_\clH(A,B)\to \Mly_\clK(FA,FB)$ --and similarly for Moore; under suitable assumptions, this can be extended to a 3-functor from (Cartesian monoidal categories, product-preserving functors, Cartesian natural transformations). This is the content of \cite[Remark 2.6]{noi:bicategories}, but it can be deduced in one fell sweep from the characterization in \cite{Katis1997}.}
  \end{enumtag}
\end{definition}
\begin{remark}\agda{Automata.agda}\label{moore_2cells}
  Unwinding the definition, an object of $\Mly(A,B)$ is a span $\inMe Eds : E \xot d A\otimes E \xto s B$,
  and a morpism $f$ from $\inMe Eds$ to $\inMe F{d'}{s'}$ is an $f : E\to F$ in the base category such that $f\circ d = d'\circ (A\otimes f)$ and $s'\circ (A\otimes f) = s$.
  Similarly, an object of $\Mre(A,B)$ is a span $\inMo Eds : E \xot d A\otimes E, E \xto s B$,
  and a morphism  $f$ from $\inMo Eds$ to $\inMo F{d'}{s'}$ is an $f : E\to F$ in the base category such that $f\circ d = d'\circ (A\otimes f)$ and $s'\circ f = s$.
\end{remark}
\begin{corollary}\agda{FMoore/Limits.agda}
  An immediate consequence of this definition (cf. \cite{noi:completeness} for the details) is that when $\clK$ is closed and $\firstblank\otimes\firstblank$ separately commutes with small colimits (e.g., when $\clK$ is closed), each category $\Mly(A,B)$ and $\Mre(A,B)$ is complete and cocomplete:
  \begin{itemize}
    \item the category $\Mre(A,B)$ has a terminal object, the internal hom $[A^*,B]$;
    \item the category $\Mly(A,B)$ has a terminal object, the internal hom $[A^+,B]$.\footnote{Assuming countable coproducts in $\clK$, the free \emph{monoid} $A^*$ on $A$ is the object $\sum_{n\ge 0} A^n$; the free \emph{semigroup} $A^+$ on $A$ is the object $\sum_{\ge 1} A^n$; clearly, if $1$ is the monoidal unit of $\otimes$, $A^*\cong 1+A^+$, and the two objects satisfy `recurrence equations' $A^+\cong A\otimes A^+$ and $A^*\cong 1+A\otimes A^*$.}
  \end{itemize}
\end{corollary}
From now on, we assume $(\clK,\times)$ is a \emph{Cartesian} category (more than often, Cartesian closed). Then, a bicategory $\Mly=\Mly_\clK$ can be defined as follows (cf. \cite{rosebrugh_sabadini_walters_1998} where this is called `$\Circ$', and a more general structure is studied, in case the base category has a non\hyp{}Cartesian monoidal structure --but then a clear analogy with Mealy machines is lost):
\begin{definition}[The bicategory $\Mly$, \cite{rosebrugh_sabadini_walters_1998}] \agda{Mealy/Bicategory.agda} \label{saba_mly}
  In the bicategory $\Mly$
  \begin{itemize}
    \item \emph{0-cells} $A,B,C,\dots$ are the same objects of $\clK$;
    \item \emph{1-cells} $A\mto B$ are the Mealy machines $(E,d,s)$, i.e. the objects of the category $\Mly(A,B)$ in \autoref{mly_pb}, thought as morphisms $\langle s,d\rangle : E\times A \to B\times E$ in $\clK$;
    \item \emph{2-cells} are Mealy machine morphisms as in \autoref{mly_pb};
    \item the composition of 1-cells $\firstblank \diamond \firstblank$ is defined as follows: given 1-cells $\langle s,d\rangle : E\times A\to B\times E$ and $\langle s',d'\rangle : F\times B \to C\times F$ their composition is the 1-cell $\langle s'\diamond s,d'\diamond d\rangle : (F\times E)\times A \to C\times (F\times E)$, obtained as
          \[\label{qui}\vxy[@C=1.5cm]{F\times E \times A \ar[r]^-{F\times \langle s,d\rangle} &
            F\times B \times E \ar[r]^-{\langle s',d'\rangle \times E} &
            C\times F\times E;}\]
    \item the \emph{vertical} composition of 2-cells is the composition of Mealy machine morphisms $f : E \to F$ as in \autoref{mly_pb};
    \item the \emph{horizontal} composition of 2-cells is the operation defined thanks to bifunctoriality of $\firstblank\diamond\firstblank : \Mly(B,C)\times \Mly(A,B)\to \Mly(A,C)$;
    \item the associator and the unitors are inherited from the monoidal structure of $\clK$.
  \end{itemize}
\end{definition}
\begin{remark}[The universal property of $\Mly$]
  Regard $\clK$ as a bicategory with a single 0-cell, and let $\Omega \clB$ be the bicategory of pseudofunctors $\bfN\to\clB$, lax natural transformations, and modifications; then (cf. \cite{Katis1997,ITA_2002__36_2_181_0}) there is an equivalence of bicategories $\Mly_\clK\cong \Omega \clK$.
\end{remark}
\begin{remark}[Composition of 1-cells fleshed out] \agda{Set/Automata.agda\#L53}
  The composition of 1-cells in $\Mly$ happens as follows (where we freely employ $\lambda$-notation available in any Cartesian closed category):
  \[
    d_2\diamond d_1 : \lambda efa.\langle d_2(f,s_1(e,a)),d_1(e,a)\rangle \label{d2d1_term}\qquad
    s_2 \diamond s_1 : \lambda efa.s_2(f,s_1(e,a))\notag
  \]
\end{remark}
\begin{definition}[Terminal extensions of machines] \agda{Set/Functors.agda\#L57} \label{termex}\leavevmode
  \begin{itemize}
    \item Every Moore machine $\inMo Eds$ can be extended as $\xymatrix{E & E\times A^* \ar[r]^-{s^\Delta}\ar[l]_-{d^\Delta}& B}$
          where $d^\Delta$ is obtained from the universal property of $A^*$, and $s^\Delta$ is defined inductively using the distributivity of products over coproducts: $s_1 = s$ and for each $n\ge 2$,
          \[s_n := \big(E\times A^n \xto{d\times A^{n-1}} E\times A^{n-1} \to \cdots \xto{d\times A} E\times A\xto{d} E \xto s B\big).\]
    \item Similarly, every Mealy machine $\inMe Eds$ has an extension $\xymatrix{E & E\times A^+ \ar[r]^-{s^\Delta}\ar[l]_-{d^\Delta} & B}$.
  \end{itemize}
\end{definition}
\begin{remark}  \label{termex_functors}
  The above procedure defines functors $(\firstblank)^\mreExt : \Mre(A,B)\to \Mly(A^*,B)$ and $(\firstblank)^\mlyExt : \Mly(A,B)\to \Mly(A^+,B)$. We call these the \emph{terminal extensions} of the Moore/Mealy machines at study (due to their link with the terminal morphism in $\Mre(A,B)$ and $\Mly(A,B)$, cf. \cite[Ch. 11]{Ehrig} and \cite{noi:completeness}).
\end{remark}
\begin{definition}[Embedding $\Mre$ in $\Mly$]\agda{Set/Adjoints.agda\#L28} \label{tautj}
  There is a 2-functor $J:\Mre\to\Mly$ from the category of Moore machines to the category of Mealy machines, regarding each Moore machine $(E,d,s)$ as a Mealy machine
  \[\vxy{
    E & E\times A\ar[l]_d\ar[r]^-{\pi_E} & E \ar[r]^s & B
    }\]
\end{definition}
\begin{notation}
  Graphical notation is very useful in bookkeeping the structure of a Mealy or Moore machine and invaluable in simplifying computations: so, let us depict the generic Mealy (red) and Moore (blue) machines with input and output $A,B$ as follows:
  \[
    \begin{tikzpicture}
      \Mealy[red!25]{E}\arrow{A}{B} \step[3]{\Moore{E}\arrow{A}{B}}
    \end{tikzpicture}
  \]
  Instead, when we need to look closer at the structure of a given automaton, we might use the following string-diagrammatic notation: the diagrams
  \[
    \begin{tikzpicture}
      \arTwo[green!25]{d} \gau{E}
      \step[1.5]{\twoAr[yellow!30]{s}\dro{B}}
      \down{\dro{A}}
      \up{\dro{E}}
    \end{tikzpicture}
  \]
  denote the same span of \eqref{moore_2cells} forming a Mealy 1-cell $(E,d,s)$ (and $d$ is oriented right-to-left!), and if the machine lies in the image of the functor $J$ above, the picture simplifies to
  \[
    \begin{tikzpicture}
      \arTwo[green!25]{d}
      \gau{E}\step[1.5]{\Umor[yellow!30]{s}\up{\dro{B}}
        \down\counit}
      \down{\dro{A}}
      \up{\dro{E}}
    \end{tikzpicture}
  \]
\end{notation}
(In such a situation, we say that $s$ is `Moore' or `of Moore type'.) Adopting this notation the composition of $(E,d_1,s_1) : A \mto B$ and $(F,d_2,s_2) : B\mto C$ is written
\[\label{da_comp}
  \begin{tikzpicture}
    \down{\twoAr[yellow!30]{s_1}\up{\gau{E}}\down{\gau A}\up[2]{\gau{F}}}
    \Uid \step{\twoAr[yellow!30]{s_2}\dro{C}}
    \begin{scope}[xshift=4cm,xscale=.75,yscale=-1,yshift=-1.5cm]
      \down\Did\comult\up[3]\comult
      \step{
        \xScale[2]{\down\Did}
        \Did
        \up[2.5]{\yScale[.5]\braid}
        \up[3]\Uid
      }
      \step[2]{
        \twoAr[yellow!30]{s_1}
        \up[3]{\twoAr[green!25]{d_1}}
      }
      \step[3]{
        \down{\twoAr[green!25]{d_2}}
        \up[2]\Uid
        \down{\dro{F}}
        \up[3]{\dro{E}}
      }
      \gau{E}\down[2]{\gau{F}}\up[3]{\gau A}
    \end{scope}
    \node at (1,-.5) {$s_2\diamond s_1$};
    \node at (5.5,-.5) {$d_2\diamond d_1$};
  \end{tikzpicture}
\]
and in case $s_2$ (resp., $s_1$) is Moore, it suitably simplifies.

From this we deduce at once the following lemma.
\begin{lemma}[Moore overrides Mealy, on both sides]\agda{Set/Automata.agda\#L58} \label{overrides}
  The composition of a Mealy 1-cell $\inMe Eds$ with a Moore 1-cell $\inMo F{d'}{s'}$ is a Moore 1-cell $\inMo {F\times E}{\#}{\#}$. Conversely, the composition of a Moore 1-cell $\inMo Eds$ with a Mealy 1-cell $\inMe F{d'}{s'}$ is a Moore 1-cell $\inMo {F\times E}{\#}{\#}$.

  In algebraic terms, the statement reads as follows: the composition bifunctors given in \eqref{qui}
  restrict to bifunctors
  \[\vxy[@R=0cm]{
      J\firstblank \diamond\firstblank : \Mre(B,C)\times\Mly(A,B) \ar[r] &\Mre(A,C)\\
      \firstblank\diamond J\firstblank : \Mly(B,C)\times\Mre(A,B) \ar[r] &\Mre(A,C).
    }\]
\end{lemma}
\begin{proof}
  Immediate using graphical reasoning and the definition of composition in \eqref{da_comp}.
\end{proof}
\begin{remark}\agda{Set/Mealyfication.agda} 
  Due to the tautological nature of $J$ `virtually all' ways to compose together a Moore and a Mealy machine agree: 
  \[\begin{array}{cccc}
      J (\fkm \diamond \fkn) = \fkm \diamond J \fkn &
      J (\fkm \diamond\fkn) = J \fkm \diamond\fkn   &
      \fkm \diamond J\fkn = J \fkm \diamond\fkn     &
      J (\fkm \diamond\fkn) = J \fkm \diamond J\fkn                                        \\
      \footnotesize\fkm\in\Mly,\fkn\in\Mre
                                                    & \footnotesize\fkm\in\Mre,\fkn\in\Mly
                                                    & \footnotesize\fkm,\fkn\in\Mre
                                                    & \footnotesize\fkm,\fkn\in\Mre
    \end{array}\]
  (after the reduction of the involved $\lambda$-terms, we obtain a strict equality.)
\end{remark}
\begin{corollary}[Moorification functors]\label{cor:prepost}
  For any two objects $A,B\in\Mly$ there exist functors
  \[\vxy[@R=0cm]{
      \firstblank\rtimes_{AB} \firstblank : \Mly(A,B)\times\Mre(A,A) \ar[r] &\Mre(A,B)\\
      \firstblank\ltimes_{AB} \firstblank : \Mre(B,B)\times\Mly(A,B) \ar[r] &\Mre(A,B)
    }\]
  Clearly then, given endo-1-cells $\fkm : A\to A$ and $\fkn : B\to B$ in $\Mre$, we can define pre- and post-composition operations as $R_\fkm=\firstblank \rtimes_{AB}\fkm$ and $L_\fkn=\fkn\ltimes_{AB}\firstblank$, which we will use in \autoref{decap_right}, \autoref{moorif_right} to characterize functors converting Mealy into Moore machines.  
\end{corollary}
\subsection{The semibicategory structure of $\Mre$}
\begin{corollary}[Sometimes you have no identity, and that's ok]
  \label{no_ids_never}
  The composition of a Moore 1-cell $[E,d,s]$ with a Moore 1-cell $[F,d',s']$ is a Moore 1-cell $[F\times E,d'\diamond d,s'\diamond s]$. Hence, if we take as objects the same of $\clK$, as 1-cells the Mealy machines of Moore type $\inMo Eds : A\mto B$, and considering 2-cells in the sense of \autoref{moore_2cells}, we obtain a sub-semibicategory of $\Mly$, the \emph{semibicategory of Moore machines}.
  It is important to note that the composition law of Moore machines is inherently nonunital: given the composition law defined in \autoref{saba_mly} \emph{there can be no other identity 1-cell than} $\inMe 1{\pi_1}{\pi_A} : A\mto A$, and this is not of Moore type.
\end{corollary}
It is worth recording the precise definition of what we denote as a semibicategory; loosely speaking, we are considering a weakening of \cite[Definition 1.1]{CTGDC_2002__43_3_163_0} in which we relax the requirement that we have strictly associative composition; a semibicategory is then the semicategorical\footnote{\emph{Semicategories} were introduced by Mitchell as `categories without identities' and studied to some extent in \cite{Mitchell1972}; \emph{semigroupal categories} are often studied in relation with their cognates the \emph{magmoidal} ones (where a mere binary functor $\firstblank\odot\firstblank : \clM\times\clM\to\clM$ is given, cf. \cite{Davydov2017nuclei}); a semibicategory with a single object, is precisely a (nonstrict) semigroupal category, cf. \cite{Lu_2018} for a proof that semigroupal categories support graphical calculi.} equivalent of a bicategory \cite{10.1007/BFb0074299}, \cite[I.3]{Gray1974}, or again equivalently the multi-object version of semigroupal categories (as defined, e.g., in \cite[3.3]{boyarchenko2007associativity}, \cite{Lu_2018}, although in a strict sense; for the weak version cf. \cite[4.1]{Elgueta2004}, or adapt \cite{CTGDC_2002__43_3_163_0} as we do). Cf. also \cite{KOCK2008,Kock2010}.
\begin{definition} \agda{Categories/SemiBicategory.agda} \label{def_semibicat} 
  A \emph{semibicategory} consists of
  \begin{enumtag}{sc}
    \item a collection of objects or 0-cells $\bfS_0$;
    \item for all objects $X,Y\in\bfS_0$ a \emph{category} $\bfS(X,Y)$;
    \item \label{bicomp} for all objects $X,Y,Z\in\bfS_0$ a \emph{composition} bifunctor $\_\circ\_ : \bfS(Y,Z)\times\bfS(X,Y)\to \bfS(X,Z)$;
    \item for all $f:Z\to W,g:Y\to Z,h:X\to Y$, an \emph{associator} isomorphism $\alpha_{fgh} : f\circ (g\circ h) \cong (f\circ g)\circ h$ (cf. \eqref{diag1}).
  \end{enumtag}
  The associator morphism must satisfy the usual \emph{pentagon equation} for all $f,g,h,k$: the diagram of \eqref{diag3} is commutative.
\end{definition}
Now, a technical result of a general kind allows replacing the semibicategory of Moore machines with a bicategory of `Moore machines', which is identical to $\Mre$ in all respects, apart from the fact that a class of symbols $\id_A$, one for each object $A\in\clK_0$, has been forcefully added in order to play the r\^ole of identity 1-cell.
\begin{theorem}\label{unitize_thm}
  The obvious forgetful 2-functor $\MonCat\to \SgCat$ admits a left adjoint $(\_)^\un : \SgCat \to \MonCat$.
\end{theorem}
\begin{proof}
  \deferredRef{unitize_thm}. Cf. also \autoref{constr_unitize}.
\end{proof}
\section{A plethora of adjoints between machines}
The key observation guiding our intuition throughout the section is the following.
\begin{lemma}[Pulling back fibered adjunctions]\agda{PullbackFiberedAdjunctions.agda}\label{fundamental_lemma}
  Given a commutative cube
  \[\label{the_cubo}\vxy[@R=3mm@C=3mm]{
      &\clM'\ar@<.25em>@{.>}[dl] \ar[rr]\ar[dd]|\hole && \clB'\ar@<.25em>[dl] \ar[dd]\\
      \clM\ar@<.25em>@{.>}[ur]\ar[rr]\ar[dd] && \clB\ar[dd] \ar@<.25em>[ur]\\
      & \clA\ar@{=}[dl] \ar[rr]|(.5)\hole && \clK\ar@{=}[dl]\\
      \clA \ar[rr]&& \clK
    }\]
  assume the front and back vertical faces are strict pullbacks in $\Cat$; then any adjunction $F:\clB \leftrightarrows \clB':G$ pulls back to an adjunction $F':\clM\leftrightarrows \clM':G'$.\footnote{A similar result holds, when the bottom horizontal square is replaced by a commutative square of adjunctions in the sense of \cite[IV.7]{McL}; we will invoke the more general version as well: refer for this to the right part of \eqref{the_real_cube}.}
\end{lemma}
\begin{proof}
  Immediate, using the 2-dimensional universal property of the pullbacks involved.
\end{proof}
\begin{corollary}
  Any adjunction $\inlinexyadj{\clK/B}{A\times\firstblank/B}{F}{G}$ fibered over $\clK$ induces an adjunction
  \[\label{pullata}\inlinexyadj{\Mre(A,B)}{\Mly(A,B)}{F}{G}\]
  fibered over $\Alg(A\times\firstblank)$. (When there is no danger of confusion, we will keep denoting as $F\dashv G$ the pulled back adjunction. Refer to the left part of \eqref{the_real_cube} for the diagram we have in mind.)
\end{corollary}
A blanket assumption for all that follows is that $\clK$ is (locally) Cartesian closed. Then there is an evident equivalence of categories $(A\times\firstblank)/B\cong\clK/B^A$; we deduce that:
\begin{lemma}\label{first_of_many}
  There are adjoint triples
  \[\label{the_triple}\vxy[@C=2cm]{
    \Sigma_f \dashv f^* \dashv \Pi_f : \clK/B^A \ar[r]|-{f^*}&\ar@<.5em>[l] \ar@<-.5em>[l] \clK/B\\
    }\]
  fibered over $\clK$, induced by a morphism $f' : B\to B^A$ (or rather by its mate $f : A\times B\to B$) pulls back to an adjoint triple
  fibered over $\Alg(A\times\firstblank)$, between $\Mre(A,B)$ and $\Mly(A,B)$.
\end{lemma}
\begin{example}
  Let $\clK$ be locally Cartesian closed; let $\pi_B : A\times B\to B$ be the projection on the second factor and $\varpi_B : B\to B^A$ its mate; it is well-known that there is a triple of adjoints $\varpi_B\circ\firstblank\dashv \varpi_B^*\dashv \Pi_{\varpi_B}$, acting respectively as composition, pullback, and direct image for $\varpi_B$ \cite[A4.1.2]{Johnstone2002}. As a consequence, we obtain a triple of functors
  \[\label{proj_triple}
    \vxy[@C=1.5cm]{
      \varpi_B\circ\firstblank \dashv\varpi_B^*\dashv \Pi_{\varpi_B} : \Mly(A,B) \ar[r] & \Mre(A,B) \ar@<.5em>[l]\ar@<-.5em>[l]
    }
  \]
\end{example}
This is not the only way to obtain adjoint situations between machines from adjoint situations between comma categories: we can exploit the fact that \autoref{fundamental_lemma} holds in slightly more generality, also for adjunctions that are not fibered over $\clK$; thus,
\begin{itemize}
  \item any natural transformation $\alpha$ between $A\times\firstblank$ and the identity functor of $\clK$ induces, by the 2-dimensional part of the universal property of comma objects, a functor $\alpha^*$ between $\clK/B$ and $(A\times\firstblank)/B\cong \clK/B^A$. But the functor $A\times\firstblank$ is a comonad on any Cartesian category, whence it's naturally copointed by the counit $\epsilon : A\times\firstblank \To \id_\clK$; from this, we get a functor
        \[\label{counitata}\vxy{
            \epsilon^* : \clK/B\ar[r] & (A\times\firstblank)/B\cong\clK/B^A\\
          }\]
        pulling back to a functor $\bar\epsilon^* : \Mre(A,B)\to \Mly(A,B)$.
  \item All morphisms inductively defined by the rule $d^{(0)}:=\pi_E$, $d^{(1)}=d : A\times E\to E$, $d^{(n)} := d \circ (A\times d^{(n-1)})$ as in \autoref{termex} induce functors
        \[\label{di_enne}\vxy[@R=0mm]{
            D_n^* : \clK/B\ar[r] & (A\times\firstblank)/B\cong\clK/B^A : s \ar@{|->}[r] & s \circ d^{(n)}
          }\]
        pulling back thanks to \autoref{fundamental_lemma} to a functor $\bar D_n : \Mre(A,B)\to \Mly(A,B)$.
\end{itemize}
Adjoint situations like \eqref{proj_triple} however provide a nifty characterization for the precomposition functors $\Mly(u,B):\Mly(A',B)\to \Mly(A,B)$ induced by $u : A\to A'$ in the base category. From $u$, using the Cartesian closed structure of $\clK$, we get $B^u : B^{A'}\to B^A$, from which, applying \autoref{fundamental_lemma} to the triple of adjoints
\[\label{triple_for_precomp_funs}\vxy[@C=2cm]{
  B^u\circ\firstblank \dashv (B^u)^* \dashv \Pi_u : \clK/B^A \ar[r]|-{(B^u)^*}&\ar@<.5em>[l] \ar@<-.5em>[l] \clK/B^{A'}\\
  }\]
(same notation of \autoref{first_of_many}) we get a triple of adjoints $\xymatrix{\Mly(A,B) \ar[r]&\ar@<.5em>[l] \ar@<-.5em>[l] \Mly(A',B)}$ and an identification of the leftmost adjoint $B^u\circ\firstblank$ with $\Mly(u,B) : \Mly(A',B)\to \Mly(A,B)$.

As a direct consequence,
\begin{remark}\label{some_adjoints_btwn_Mly}
  Let $\clK$ be locally Cartesian closed; then each functor $\Mly(u,B) : \Mly(A',B)\to \Mly(A,B)$ sits in a triple of adjoints $\Mly(u,B) \dashv (B^u)^* \dashv \Pi_u$.
\end{remark}
\color{black}
\begin{remark}[Adjunctions between hom-categories]\label{some_adjs}
  In case $\clK$ is locally presentable (e.g., $\clK=\Set$, any Grothendieck topos, a docile category of spaces, etc.) AFTs in their classical form \cite[Ch. 5]{Bor2}, \cite[3.3.3---3.3.8]{Bor1} applied to $\epsilon^*, D_n^*$ yield right adjoints $R$ and $K_n$ and thus, by virtue of \autoref{fundamental_lemma} above, we get
  \begin{itemize}
    \item An adjunction $\bar \epsilon^* : \Mre(A,B)\rightleftarrows \Mly(A,B) : \bar R$;
    \item for evey $n\ge 1$ an adjunction $\bar D_n : \Mre(A,B)\rightleftarrows \Mly(A,B)  : K_n$, with $D_n$ fully faithful --and thus $\bar K_n$ is a coreflector.
  \end{itemize}
\end{remark}
Among these, especially the second bears some importance to automata theory. We will analyse its structure, providing a clean explicit construction for $K_0, K_1$ in the rest of the present section; in particular, we will completely characterise the coreflective subcategory of $\bar D_n\bar K_n$-coalgebras. A crucial point of our discussion will be that $\bar K_n$ arises as composition $\fku\ltimes\firstblank$ for a universal 1-cell $\fku$ of Moore type, following the notation of \autoref{cor:prepost}.
\begin{remark}[Biadjunctions between $\Mly$ and $\Mre$]
  The adjunctions in \autoref{some_adjs} constitute the action on hom-categories of \emph{biadjunctions} (cf. \cite{betti1981quasi} for an earlier reference, and the notion of a \emph{left adjoint of a left lax transformation} in \cite[p. 8]{CTGDC_1972__13_3_217_0}. This is a weakening of \cite[Ch. 9]{fiore} that drops the request to have adjoint \emph{equivalences} of hom-categories, replacing them with a family of adjunctions indexed over the objects of the two bicategories) $\bsE\dashv \bsR$, $\{\bsD_n\dashv \bsK_n \mid n\ge 0\}$
  between the semibicategory $\Mre$ (or equivalently its bicategorical reflection $\Mre^\un$ in the sense of \autoref{unitize_thm}) and the bicategory $\Mly$ of Mealy machines of \autoref{saba_mly}. Each of these (2-semi) functors is the identity on 0-cells.
\end{remark}
\begin{definition}
  In the notation of \autoref{di_enne}, the functors $D_0$ and $D_1$ are defined respectively as follows: a generic Moore machine $\inMo Eds$ goes to $\inMe Ed{s\pi_E}$ and $\inMe Ed{sd}$, and more in general, $D_n$ is the upper horizontal functor in 
  \[\vxy{
    \clK/B \ar[d]_U \ar[r]^-{D_n}& \clK/B^A \ar[d]^{U'}\\
    \clK \ar[r]_-{A^n\times\firstblank}& \clK
    }\]
  adapted from \autoref{fundamental_lemma} (In the lower rows, the functors $X\mapsto A^n\times X$ are clearly left adjoints).
\end{definition}
\begin{definition}[The universal Moore machine]\agda{Set/LimitAutomata.agda\#L65}\label{the_univ_moo}
  Let $X\in\clK_0$ be an object. There exists a Moore machine $\fku X = \inMo X{\pi_X}{\id_X}$ in $\Mre(X,X)$ where $\pi_X : X\times X\to X$ is the projection on the first coordinate. 
\end{definition}
The machine $\fku X$ can't be the identity 1-cell in $\Mre(X,X)$ by \autoref{no_ids_never} above, but it is `as near to the identity as a Moore 1-cell can be' to the identity in $\Mly(X,X)$, in a sense made precise by the following result: post-composition with $\fku X$ realizes the Moore machine associated to a Mealy machine.
\begin{theorem}[Moorification is composition with a universal cell]\label{moorif_right}\agda{Set/Adjoints.agda\#L28}
  The functor $K_1$ of \autoref{some_adjs} is naturally isomorphic to the functor $L_{\fku B}=\fku B\ltimes\firstblank : \Mly(A,B)\to \Mre(A,B)$ of \autoref{cor:prepost}.
\end{theorem}
\begin{remark}
  In layman's terms, $\fku B=\inMo{B}{\pi_B}{\id_B}$ is the \emph{universal machine of Moore type}, in the sense that sequential composition with it turns every Mealy machine into an `equivalent' Moore machine.

  This is a clean-cut categorical counterpart for the extremely well-known fact that the two concepts of Moore and Mealy machines are `essentially equivalent' in a sense made precise by \cite[3.1.4, 3.1.5]{Shallit}.
\end{remark}
\begin{construction}\agda{Set/LimitAutomata.agda\#L52}\label{constr_pinfty}
  Let $X\in\clK_0$ be an object. We define a Moore machine $\fkp X=\inMo {P_\infty X}{\delta_X}{\sigma_X}$ in $\Mre(X,X)$ as follows:
  \begin{itemize}
    \item its carrier results from the pullback in the base category $\clK$
          \[\label{pinfty}
            \vxy{
            P_\infty X \ar[r]\ar[d]\xpb & 1 \ar[d]^\h \\
            [X^*,X] \ar[r]& [X^+,X]
            }\]
          where the lower horizontal map results from the obvious inclusion $A^+\hookrightarrow A^*=1+A^+$ and $\h$ picks the element $\lambda(a\kons as).a$.\footnote{In every Cartesian closed category one obtains $\h : X^+\to X$ from the family $\pi_1 : X^n\to X$ of projections on the first coordinate, thanks to the universal property of $X^+=\sum_{n\ge 1} X^n$.}
    \item the map $\delta_X : P_\infty X\times X\to P_\infty X$ is defined by restriction to the dynamics of $[X^*,X]$ as $(f,a)\mapsto f$;
    \item the map $\sigma_X : P_\infty X \to X$ is defined as $f\mapsto f[\,]$.
  \end{itemize}
\end{construction}
\begin{remark}\label{pinfty_notation}
  Informally speaking, the carrier of $\fkp X$ is the subobject $H$ of $[X^*,X]$\footnote{Indeed, the left vertical map is monic; for the sake of clarity we write as if $\clK=\Set$, leaving the straightforward element-free rephrasing to the conscientious readers.}  made of those functions $f : X^*\to X$ that on nonempty lists coincide with the head function: $f[\,]$ is free to assume any value whatsoever, while $f(a\kons as)=a$ for every nonempty list $a\kons as$.

  Given this, $\sigma : f\mapsto f [\,]$ is clearly a bijection. So, it might seem unnecessarily verbose to distinguish $\fku X$ and $\fkp X$; yet, $\sigma$ is \emph{not} a Moore machine homomorphism, as the diagrams of \eqref{moore_2cells} do not commute when filled with $\sigma$.
\end{remark}
A property that is stable under isomorphism, enjoyed by $P_\infty A$ but not $A_\infty$ is the following: the diagram
\[\vxy{
  X\times P_\infty X\ar[r]^{\delta_X} \ar[d]_{\pi_{P_\infty X}}& P_\infty X \ar[d]^s \\
  P_\infty X \ar[r]_s & X
  }\]
commutes for $P_\infty X$ but not for $X$; the same property enforced for $\fku X$ would entail that $\forall(x,x')\in X.x=x'$, a blatantly false statement.
\begin{definition}[Soft machine, \protect{\cite{burroughs1961soft}}]\agda{Set/Soft.agda}\label{def_acep}
  A Moore machine $\inMo Eds$ is called \emph{soft} if the following square commutes:
  \[\label{diag_aceph}\vxy{
      A\times E\ar[r]^-d\ar[d]_{\pi_E} & E\ar[d]^s \\
      E \ar[r]_-s& B
    }\]
  and in simple terms, it can be thought of as a Moore machine whose output map `skips a beat and then starts computing', as its terminal extension $s^\Delta$ (cf. \autoref{termex}) is such that $s^\Delta(a\kons as,e) = s^\Delta(as,e)$.
\end{definition}
Clearly, soft Moore machines $A\mto B$ form a full subcategory $\aMre(A,B)$ of $\Mre(A,B)$; we denote $i : \aMre\to\Mre$ the most obvious locally fully faithful inclusion 2-functor. More is true:
\begin{lemma}[Soft machines cut other heads]\agda{Set/Soft.agda\#L32}\label{cut_head}
  Let $\fkm : A\mto B$ and $\fkn : B\to C$ be a Mealy machine and a Moore machine; if $\fkn$ is soft, then the composition $J\fkn\diamond \fkm$ is also soft.\footnote{Warning: it is \emph{not} true in general that the composition $\fkm\diamond J\fkn$ is soft.} Translated in graphical terms, the statement reads as follows, if we depict a soft Moore machine as \raisebox{.25em}{\tikz{\Moore[cyan!25]{\fkn}\fill ($(.75,.25)+(-.05,.075)$) circle (1pt);}}:
  \[
    \begin{tikzpicture}
      \Mealy{}\step{\Moore[cyan!25]{}\fill ($(.75,.25)+(-.05,.075)$) circle (1pt);}
      \step[3]{\Moore[cyan!25]{}\fill ($(.75,.25)+(-.05,.075)$) circle (1pt);}
    \end{tikzpicture}
  \]
\end{lemma}
\begin{proof}
  Immediate by inspection, using graphical reasoning (diagram \eqref{diag_aceph} is easy to translate graphically) and the definition of composition in \autoref{da_comp}.
\end{proof}
\begin{remark}\agda{Set/Soft.agda\#L40}\label{P_is_acep}
  The Moore machine $\inMo{P_\infty X}{\delta_X}{\sigma_X}$ in \autoref{constr_pinfty} is soft.
\end{remark}
\begin{corollary}\agda{Set/Soft.agda\#L43}
  Let $\inMe Eds : A\mto B$ be a generic Mealy machine; then, by virtue of \autoref{overrides}, the composition $\inMo{P_\infty B}{\delta_B}{\sigma_B}\,\ltimes_{AB}\,\inMe Eds$ is a Moore machine, and by virtue of \autoref{cut_head} it is soft. Hence, the functor $L_{\fkp B} : \Mly(A,B) \to \Mre(A,B)$ factors through the obvious full inclusion $i_A : \aMre\to\Mre$.
\end{corollary}
\begin{remark}\label{ace_as_coalg}
  Given that $D_1$ is fully faithful, the unit of the adjunction $D_1\dashv K_1$ is a natural isomorphism; denote $S_1 := D_1K_1$ the comodality (= idempotent comonad) obtained as a consequence. Then the category of $S_1$-coalgebras is isomorphic to the category of soft Moore machines.
\end{remark}
An object of the form $\inMo{P_\infty A}{\delta_A}{\sigma_A}$ is thus a universal soft Moore machine, in the sense specified by the following coreflection theorem:
\begin{theorem}[Decapitation as a right adjoint]\label{decap_right}\agda{Set/Adjoints.agda\#L46}
  The coreflector of the comodality of \autoref{ace_as_coalg} is naturally isomorphic to the postcomposition functor $L_{\fkp B}=\fkp B\ltimes\firstblank$.
\end{theorem}
Similarly, denote $S_n=D_nK_n$ the comodality of the adjunction obtained from \autoref{di_enne}. Then, the category of $S_1$-coalgebras is isomorphic to the full category $\aMre[n]$ of \emph{$n$-soft Moore machines}, if we understand \autoref{def_acep} generalised as follows: the square
\[\label{diag_naceph}\vxy{
  A^n\times E\ar[r]^-{d^{(n)}}\ar[d]_{\pi_E} & E\ar[d]^s \\
  E \ar[r]_-s& B
  }\]
commutes (cf. \eqref{di_enne} for the definition of $d^{(n)}$).

Evidently, $\aMre = \aMre[1]\subseteq \aMre[2]$ and more in general, there is a chain of full inclusions
\[\vxy{
		\aMre\,\, \ar@{^{(}->}[r]& \,\,\aMre[2]\,\, \ar@{^{(}->}[r]& \,\,\aMre[3]\,\, \ar@{^{(}->}[r]& \,\,\dots\,\, \ar@{^{(}->}[r]& \,\,\Mre
	}\]
analogous to the `level' filtration of a topos (cf. \cite{kelly1989complete,Kennett2011,menni2019level,roy1997ball,Street2000}, especially in the case of simplicial sets or geometric shape for higher structures indexed over the natural numbers).

As a final remark, let's record that the results of \autoref{some_adjoints_btwn_Mly} provide an extremely intrinsic characterization for the object $P_\infty X$ of \autoref{constr_pinfty}, and a comparison theorem for the operations $(\firstblank)^\mreExt,(\firstblank)^\mlyExt$ of terminal extension of a Mealy machine and of a Moore machine expounded in \autoref{termex_functors}.
\begin{lemma}\agda{Set/Functors.agda\#L113}
  The left square in \eqref{sq_of_J}
  commutes (in the notation of \eqref{triple_for_precomp_funs}, $\Mly(t,B)$ is the obvious restriction functor induced by the `toList' map $t : A^+\hookrightarrow A^*$).
\end{lemma}
\begin{proof}
  Immediate inspection.
\end{proof}
\begin{corollary}[Corollary of \autoref{some_adjoints_btwn_Mly}]
  The functor $\Mly(t,B)$ sits in a triple of adjoints $\Mly(t,B) \dashv (B^t)^* \dashv \Pi_t$, where $(B^t)^*$ is given pulling back a morphism $\var{X}{B^{A^+}}$ along $B^t : B^{A^*}\to B^{A^+}$.
\end{corollary}
It follows from an easy inspection that $B^t : [A^*,B]\cong B\times [A^+,B]\to [A^+,B]$ coincides with the projection map, and thus we can consider the right square in \eqref{sq_of_J},
obtained mating the left square, and prove at once that
\begin{proposition}
  The left square of \eqref{sq_of_J} is exact, in the sense of \cite{guitart:1980,Guitart1981,Pavlovi1991} or \cite[1.8.9.(b)]{CLTT}. This is to say, $\theta$ in the right diagram \emph{ibid.} is invertible.
\end{proposition}
Now let $\h : 1\to [A^+,A]$ be the morphism in \eqref{pinfty}: from \autoref{some_adjoints_btwn_Mly} we get a functor $\h^*$ of pullback along $\h$ (with both a left and a right adjoint), and an isomorphism between the Moore machine $\fkp X$ and $(A^t)^*(\h)$.
\section{Future plans}
Among many, we are left with three questions:
\begin{itemize}
	\item What is the filtered colimit of the chain $\aMre[n]\to \aMre[n+1]$, and are these inclusions coreflective as well?
	\item Can we find a representation for \emph{all} $K_n$ in terms of postcomposition with universal 1-cells? In other words, can we find the universal $n$-soft machine, perhaps adapting \autoref{constr_pinfty}?
	\item \autoref{moorif_right} and \autoref{decap_right} characterize $K_0,K_1$ as postcomposition; but then $D_0,D_1$ are \emph{left liftings} in $\Mly$: can we deduce something meaningful from this intrinsic characterization?\footnote{Observe that the existence of right liftings \cite[2.1.3]{coend-calcu} can be ruled out by the fact that postcompositions do not preserve colimits in general.}
\end{itemize}

\printbibliography
\appendix
\section{Diagrams and Proofs}
\subsection{Diagrams}
\[\label{diag1}\vxy{
	\clS_0\times\clS_0\times\clS_0 \ar[d]_-{\otimes\times\clS_0} \ar[r]^-{\clS_0\times\otimes} \drtwocell<\omit>{\alpha} & \clS_0\times\clS_0 \ar[d]^-\otimes \\
	\clS_0\times\clS_0 \ar[r]_-\otimes & \clS_0
	}\]
\def\c{\circ}
\[\label{diag3}\vxy[@C=-9mm]{
		& f\c (g\c (h \c k))) \ar@{->}[rr] \ar@{->}[ld] &  & (f\c g)\c (h \c k) \ar@{->}[rd] &  \\
		f\c ((g\c h)\c k) \ar@{->}[rrd] &  &  &  & ((f\c g)\c h)\c k \\
		&  & (f\c (g\c h))\c k \ar@{->}[rru] &  &
	}\]
	\[\label{the_real_cube}\vxy[@R=3mm@C=3mm]{
		&\Mly(A,B)\ar@<.25em>@{.>}[dl] \ar[rr]\ar[dd]|\hole && A\times\firstblank/B\ar@<.25em>[dl] \ar[dd]\\
		\clM\ar@<.25em>@{.>}[ur]\ar[rr]\ar[dd] && \clK/B\ar[dd] \ar@<.25em>[ur]\\
		& \Alg(A\times\firstblank)\ar@{=}[dl] \ar[rr]|(.525)\hole && \clK\ar@{=}[dl]\\
		\Alg(A\times\firstblank) \ar[rr]&& \clK
	}\kern3em
	\vxy[@R=3mm@C=3mm]{
		&\Mly(A,B)\ar@<.25em>@{.>}[dl] \ar[rr]\ar[dd]|\hole && A\times\firstblank/B\ar@<.25em>[dl] \ar[dd]\\
		\clM\ar@<.25em>@{.>}[ur]\ar[rr]\ar[dd] && \clK/B\ar[dd] \ar@<.25em>[ur]\\
		& \Alg(A\times\firstblank)\ar@<.25em>[dl]\ar@{<-}@<-.25em>[dl] \ar[rr]|(.525)\hole && \clK\ar@<.25em>[dl]\ar@{<-}@<-.25em>[dl]\\
		\Alg(A\times\firstblank) \ar[rr]&& \clK
	}\]
	\[\label{sq_of_J}
    \vxy{
      \Mre(A,B) \ar[r]^{(\firstblank)^\mreExt}\ar[d]_{D_1}& \Mly(A^*,B)\ar[d]^{\Mly(t,B)}\\
      \Mly(A,B) \ar[r]_{(\firstblank)^\mlyExt}& \Mly(A^+,B) \\
    }\qquad\qquad 
	\vxy{
		\Mre(A,B) \ar[r]^{(\firstblank)^\mreExt}\ar@{<-}[d]_{K_1}& \Mly(A^*,B)\ar@{<-}[d]^{(B^t)^*}\dltwocell<\omit>{\theta}\\
		\Mly(A,B) \ar[r]_{(\firstblank)^\mlyExt}& \Mly(A^+,B) \\
	  }\]
\subsection{Proofs}
This is essentially the argument that appears in passing in \cite{Lu_2018}, extended from \emph{strict} semigroupal/monoidal categories to \emph{non-strict} ones.
\begin{construction}\label{constr_unitize}
	Let $\clS$ be a semigroupal category with underlying category $\clS_0$ and equipped with a bifunctor $\_\otimes\_ : \clS_0\times\clS_0 \to \clS_0$
	satisfying the associativity axiom. We define $\clS^\un$ as follows:
	\begin{itemize}
		\item the underlying category $(\clS^\un)_0$ is the coproduct $\clS_0 + \bf 1$, where $\bf 1$ is the terminal category with unique object $\perp$;
		\item the `extended' multiplication functor
		      \[\vxy{\_\otimes^\un\_ : (\clS_0+{\bf 1})\times(\clS_0 +{\bf 1}) \cong (\clS_0\times\clS_0) + (\clS_0\times {\bf 1}) + ({\bf 1}\times\clS_0) + ({\bf 1}\times {\bf 1}) \ar[r]& \clS_0+{\bf 1}}\]
		      is defined piecewise as $S\otimes^\un S' = S\otimes S'$ if $S,S'\in\clS_0$ and $\perp\otimes^\un S' = S'$, $\perp\otimes^\un \perp=\perp$.
		      (Notice in particular that in case $\clS$ had a
		      monoidal unit, $\perp$ `replaced' it: we have added $\perp$ as a free unit.) On morphisms we follow a similar strategy; there is only an identity in $\bf 1$, and no morphism $\perp \leftrightarrows S$, so we just have to define $f\otimes^\un g := f\otimes g$ and $\id_\perp \otimes^\un \id_\perp = \id_\perp$, $f\otimes^\un \id_\perp=f$, $\id_\perp\otimes^\un g$ are forced to be $f$ and $g$ respectively, if we want that $\otimes^\un$ sends identities to identities, and that is is bifunctorial. Thus, $\perp\otimes^\un\_$ and $\_\otimes^\un\perp$ are \emph{strictly} the identity functors of the category $(\clS^\un)_0$.
	\end{itemize}
\end{construction}
\begin{proof}[Proof that the functor in \autoref{unitize_thm} is a left adjoint]\label{proof_of_unitize_thm}
	First of all let's specify the last (straightforward) piece of structure needed to define the unitization.
	\begin{itemize}
		\item The associator has components $\alpha_{XYZ}$ on objects of $\clS_0$, and when either $X,Y$ or $Z$ is $\perp$ we define it to be the appropriate identity morphism of the tensor of the remaining two objects. Naturality of $\alpha$ is ensured by these choices. This also ensures that the pentagon axiom for 4 objects $X,Y,Z,W$ trivially holds either because $X,Y,Z,W$ are all in $\clS_0$ (and thus the pentagon axiom is true in the semigroupal structure of $\clS_0$) or because the pentagon is made of identities.
		\item The last bit of structure that we have to assess is the `triangle axiom' for a monoidal structure: the triangle
		      \[\vxy{
			      X \otimes^\un (\perp\otimes^\un Y) \ar@{->}[rr]^{\alpha_{X\perp Y}} \ar@{->}[rd] &  & (X \otimes^\un \perp)\otimes^\un Y \ar@{->}[ld] \\
			      & X\otimes^\un Y &
			      }\]
		      must commute as it is composed by identities, and reasoning with a similarly straightforward case analysis, we conclude it does.
	\end{itemize}
	We have exhibited a construction for the unitization of a semigroupal category $\clS$; we still have to prove that this is a functor (but this is obvious: each semigroupal functor $H : \clS\to \clS'$ induces a monoidal functor coinciding with $H$ on $\clS_0$ and sending $\perp_\clS$ to $\perp_{\clS'}$), and its universal property. For the latter, we have to check that there is an isomorphism
	\[\SgCat(\clS,\forg \clM)\cong\MonCat(\clS^\un,\clM)\]
	between the category of strong semigroupal functors $\clS\to \forg \clM$ into a monoidal category, and the category of strong monoidal functors $\clS^\un\to\clM$, when both categories are equipped with semigroupal and monoidal natural transformations respectively.

	At the level of the underlying categories, $(\_)^\un$ acts as the `Maybe' functor, and $\clM_0$ is a pointed object in $\Cat$ (by the monoidal unit, $I_\clM : {\bf 1}\to\clM$) so that $H : \clS\to \forg \clM$ induces a unique functor $\hat H : \clS_0 + {\bf 1}\to\clM_0$; this functor is now preserving the tensor product on $\clS_0$ (because $H$ was semigroupal to start with), and it is forced by the universal property of `Maybe' to send $\perp$ into $I_\clM$, so strictly speaking it becomes a \emph{normal} strong monoidal functor (=strictly preserving the identity).
\end{proof}
\begin{proof}[Proof of \autoref{unitize_thm}]\label{proof_of_thm_freebicat}
	Let $\bfS=(\bfS_0,\_\circ\_)$ be a semibicategory as in \autoref{def_semibicat}; define $\free\bfS$ as follows:
	\begin{itemize}
		\item $(\free\bfS)_0$ is the same class of objects of $\bfS$;
		\item for each $X\in\bfS_0$, we apply to the category $\bfS(X,X)$ the construction of \ref{constr_unitize}; thus, each semigroupal category $\bfS(X,X)$ is unitized into $\bfS(X,X)^\un$.
		\item the composition functors are defined as in $\bfS$ when $X,Y,Z$ are such that neither $Y=Z$ or $X=Y$, which means that in such case we take the same
		      \[\_\circ_{XYZ}\_ : \xymatrix{\bfS(Y,Z)\times\bfS(X,Y)\ar[r] & \bfS(X,Z)}\]
		      as composition maps. When $X=Y$ instead we define
		      \[\_\circ_{\un,XXZ}\_ : \xymatrix{
				      \bfS(X,Z)\times(\bfS(X,X)+{\bf 1})\cong
				      \bfS(X,Z)\times \bfS(X,X)+\bfS(Y,Z)\times {\bf 1}
				      \ar[r] & \bfS(X,Z)}\]
		      via the coproduct map of $\_\circ_{XXZ}\_$ and the right unitor of $(\Cat,\times)$. We act similarly to define
		      \[\_\circ_{\un,XYY}\_ : \xymatrix{
			      (\bfS(Y,Y) + {\bf 1})\times\bfS(X,Y)
			      \cong
			      \bfS(Y,Y)\times\bfS(X,Y) + {\bf 1}\times\bfS(X,Y)
			      \ar[r] & \bfS(X,Y)}\]
		      using the coproduct map of $\_\circ_{XYY}\_$ and the left unitor of $(\Cat,\times)$.

		      Clearly, all $\_\circ_{XYZ}\_$ so defined are functors.
		\item The components of the associator $\alpha^\un$ when either $f,g$ or $h$ is the `dummy identity' in a $\bfS(X,X)$ reduce to the identity 2-cells of the remaining two objects, considering how we defined the composition operation; on all other components, $\alpha^\un = \alpha$. Again, these choices ensure the naturality of $\alpha^\un$.
		\item Clearly, any strong semigroupal functor $H:\bfS \to \bfT$ induces a strong monoidal unctor $H^\un : \bfS^\un \to \bfT^\un$, that coincides with $H$ on $\bfS\hookrightarrow \bfS^\un$.
	\end{itemize}
	It remains to show the universal property. A pseudofunctor $H : \bfS \to \bfB$ from a semibicategory to a bicategory must induce a pseudofunctor $\bfS^\un\to\bfT$: at the level of objects nothing has changed; each hom-functor
	\[\vxy{
			H_{XY} : \bfS(X,Y) \ar[r] & \bsB(HX,HY)
		}\]
	with $X\ne Y$ is also unchanged, and since each $\bfB(HX,HX)$ is a monoidal category (hence, in particular, a pointed category) the functor $H$ extends to $H_{XX}^\un : \bfS(X,X)^\un\to\bfB(HX,HX)$; it is easily seen that this extension is to a strong monoidal functor.
\end{proof}
\end{document}